\documentclass{imanumARXIV}
\usepackage{graphicx}
\usepackage{subfig}
\usepackage{verbatim}

\newcommand{\norm}[1]{\| #1 \|}

\begin{document}

\title{Analysis of the discontinuous Galerkin method for elliptic problems on surfaces}

\author{%
{\sc
Andreas Dedner,
Pravin Madhavan
{\sc and}
Bj\"orn Stinner} \\[2pt]
Mathematics Institute and Centre for Scientific Computing, University of Warwick,\\
Coventry CV4 7AL, UK
}

\maketitle

\begin{abstract}
{
We extend the discontinuous Galerkin (DG) framework to a linear
second-order elliptic problem on a compact smooth connected and oriented
surface in $\mathbb{R}^{3}$. An interior penalty (IP) method is introduced
on a discrete surface and we derive a-priori error estimates by relating
the latter to the original surface via the lift introduced in
\cite{dziuk1988finite}. The estimates suggest that the geometric error
terms arising from the surface discretisation do not affect the overall
convergence rate of the IP method when using linear ansatz functions. 
This is then verified numerically for a number of test problems.
An intricate issue is the approximation of the surface conormal required in the IP formulation, choices of which are investigated numerically. Furthermore, we present a generic implementation of test problems on surfaces.
}
{discontinuous galerkin; interior penalty; surface partial differential equations; error analysis.}
\end{abstract}

\section{Introduction}
\label{sec:Introduction}
Partial differential equations (PDEs) on manifolds have become an active area of research in recent years due to the fact that, in many applications, models have to be formulated not on a flat Euclidean domain but on a curved surface. For example, they arise naturally in fluid dynamics (e.g.~surface active agents on the interface between two fluids, \cite{JamLow04}) and material science (e.g.~diffusion of species along grain boundaries, \cite{DecEllSty01}) but have also emerged in areas as diverse as image processing and cell biology (e.g.~cell motility involving processes on the cell membrane, \cite{neilson2011modelling} or phase separation on biomembranes, \cite{EllSti10}). 

Finite element methods (FEM) for elliptic problems and their error analysis have been successfully applied to problems on surfaces via the intrinsic approach in \cite{dziuk1988finite} based on interpolating the surface by a triangulated one. This approach has subsequently been extended to parabolic problems in \cite{dziuk2007surface} as well as evolving surfaces in \cite{dziuk2007finite}. \cite{ju2009finite} and \cite{GieMue_prep} have also considered finite volume methods on surfaces via the intrinsic approach.   
However, as in the planar case there are a number of situations where FEM may not be the appropriate numerical method, for instance, advection dominated problems which lead to steep gradients or even discontinuities in the solution. 

DG methods are a class of numerical methods that have been successfully applied to hyperbolic, elliptic and parabolic PDEs arising from a wide range of applications. Some of its main advantages compared to `standard' finite element methods include the ability of capturing discontinuities as arising in advection dominated problems, and less restriction on grid structure and refinement as well as on the choice of basis functions. 
The main idea of DG methods is not to require continuity of the solution between elements. Instead, inter-element behaviour has to be prescribed carefully in such a way that the resulting scheme has adequate consistency, stability and accuracy properties. A short introduction to DG methods for both ODEs and PDEs is given in \cite{cockburn2003discontinuous}. A history of the development of DG methods can be found in \cite{cockburn2000development} and \cite{arnold2002unified}. \cite{arnold2002unified} provides an in-depth analysis of a large class of discontinuous Galerkin methods for second-order elliptic problems.

The motivation of this study has been to investigate the issues arising when attempting to apply DG methods to problems on surfaces. We restrict our analysis to a linear second-order elliptic PDE on a compact smooth connected and oriented surface. We expect that parabolic problems on evolving surfaces as featured in the above mentioned applications can be dealt with along the lines of \cite{dziuk2007finite}. 

This paper is organised in the following way. We consider a second-order elliptic equation on a compact smooth connected and oriented surface $\Gamma \subset \mathbb{R}^{3}$ and introduce a particular DG method known as the interior penalty (IP) method on a triangulated surface $\Gamma_{h}$. The surface IP method we consider is similar in nature to the one introduced in \cite{arnold1982interior}, and its well-posedness follows naturally from results in the planar case given in \cite{arnold2002unified} and \cite{ainsworth2011constant}. We then derive a-priori error estimates in the appropriate norms by relating $\Gamma_{h}$ to $\Gamma$ via a lifting operator and by making use of results from \cite{dziuk1988finite} and \cite{GieMue_prep} to show that the additional geometric error terms arising when approximating the surface scale in such a way that they do not affect the convergence rates proved and observed for the standard FEM approach in \cite{dziuk1988finite} when using linear ansatz functions. 
     
We then present some numerical results, making use of the Distributed and Unified Numerics Environment (DUNE) software package (see \cite{dunegridpaperII:08}, \cite{dunegridpaperI:08}) and, in particular, the DUNE-FEM module described in \cite{dunefempaper:10} (also see dune.mathematik.uni-freiburg.de for more details on this module). We consider a number of test problems, for which we compute experimental orders of convergence (EOCs) in both the $L^{2}$ norm and the $DG$ norm, and show that these coincide with the theoretical error estimates derived in the previous section. Furthermore, we consider several intuitive ways of approximating the surface conormal in our IP formulation, and investigate the resulting schemes numerically. In the process, we present a generic implementation of test problems on surfaces which follows as a direct application of the \cite{demlow2008adaptive} algorithms. 

Finally, we briefly present numerical results for nonconforming grids and higher order polynomial ansatz functions, which suggest that the convergence rates of the standard FEM approach still hold for such generalisations. 

\section{Notation and Setting}
\label{sec:NotationAndSetting}
The notation in this section closely follows the one used in \cite{dziuk1988finite}.
Let $\Gamma$ be a compact smooth connected and oriented surface in $\mathbb{R}^{3}$. For simplicity, we assume that $\partial \Gamma = \emptyset$. Let $d$ denote the signed distance function to $\Gamma$ which we assume to be well-defined in a sufficiently thin open tube $U$ around $\Gamma$. The orientation of $\Gamma$ is set by taking the normal $\nu$ of $\Gamma$ to be in the direction of increasing $d$ whence
\[\nu(\xi) = \nabla d(\xi),\ \xi \in \Gamma. \]
With a slight abuse of notation we also denote the projection to $\Gamma$ by $\xi$, i.e.~$\xi:U \rightarrow \Gamma$ is given by
\begin{equation}\label{eq:uniquePoint}
\xi(x) = x - d(x)\nu(x) \quad \mbox{where } \nu(x):=\nu(\xi(x)).
\end{equation}
Later on, we will consider a triangulated surface $\Gamma_h \subset U$ approximating $\Gamma$ such that there is a one-to-one relation between points $x \in \Gamma_h$ and $\xi \in \Gamma$ so that, in particular, the above relation (\ref{eq:uniquePoint}) can be inverted. Throughout this paper, we denote by
\[ P(\xi):= I - \nu(\xi)\otimes \nu(\xi),\ \xi\in \Gamma, \]
the projection onto the tangent space $T_{\xi}\Gamma$ on $\Gamma$ at a point $\xi \in \Gamma$. Here $\otimes$ denotes the usual tensor product.
\begin{definition}
For any function $\eta$ defined on an open subset of $U$ containing $\Gamma$ we can define its \emph{tangential gradient} on $\Gamma$ by
\[ \nabla_{\Gamma}\eta := \nabla \eta - \left(\nabla \eta\cdot \nu \right) \nu = P\nabla \eta\]
and then the \emph{Laplace-Beltrami} operator on $\Gamma$ by 
\[ \Delta_{\Gamma} \eta := \nabla_{\Gamma}\cdot (\nabla_{\Gamma} \eta).
\]
\end{definition}
\begin{definition}
We define the surface Sobolev spaces
\[ H^{m}(\Gamma) := \{u \in L^{2}(\Gamma) \ : \ D^{\alpha}u \in L^{2}(\Gamma)\ \forall |\alpha| \leq m \}, \quad m \in \mathbb{N} \cup \{ 0 \}, \]
with corresponding Sobolev seminorm and norm respectively given by
\[ 
|u|_{H^{m}(\Gamma)} := \left(\sum_{|\alpha|=m} \norm{D^{\alpha}u}_{L^{2}(\Gamma)}^{2}\right)^{1/2}, \quad 
\norm{u}_{H^{m}(\Gamma)} := \left(\sum_{k=0}^m |u|_{H^{k}(\Gamma)}^{2}\right)^{1/2}.
\]
\end{definition}     
We refer to \cite{wlokapartial} for a proper discussion of Sobolev spaces on manifolds. 

The problem that we consider in this paper is the following second-order elliptic equation:
\begin{equation}\label{eq:EllipticGamma}
-\Delta_{\Gamma} u + u = f
\end{equation}
for a given $f \in L^{2}(\Gamma)$. Using integration by parts on surfaces the weak problem reads:
\\
\\
$\big(\mathbf{P}_{\Gamma}\big)$ Find $u \in H^{1}(\Gamma)$ such that  
\begin{equation} \label{eq:weakH1}
 \int_\Gamma \nabla_{\Gamma}u\cdot \nabla_{\Gamma}v + u v\ dA =
   \int_\Gamma fv\ dA \quad\forall v\in H^1(\Gamma). 
\end{equation}
Existence and uniqueness of a solution $u$ follows from standard arguments. We assume that $u \in H^{2}(\Gamma)$ satisfies 
\begin{align}\label{eq:EllipticRegularity}
\norm{u}_{H^{2}(\Gamma)} \leq C\norm{f}_{L^{2}(\Gamma)}
\end{align}
where we refer to \cite{aubin1982nonlinear} and \cite{wlokapartial} for more details on elliptic regularity on surfaces. 

\section{Approximation and Properties} \label{sec:ApproximationAndProperties}
To obtain a discretisation of $u$, the smooth surface $\Gamma$ is approximated by a polyhedral surface $\Gamma_{h} \subset U$ composed of planar triangles. Let $\mathcal{T}_{h}$ be the associated regular, conforming triangulation of $\Gamma_{h}$ i.e.
\[ \Gamma_{h} = \bigcup_{K_{h} \in \mathcal{T}_{h}} K_{h}. \]
The vertices are taken to sit on $\Gamma$ so that $\Gamma_{h}$ is its linear interpolation. We assume that the projection map $\xi$ defined in (\ref{eq:uniquePoint}) is a bijection when restricted to $\Gamma_h$, thus avoiding multiple coverings of $\Gamma$ by $\Gamma_h$. 
Let $\mathcal{E}_{h}$ denote the set of all codimension one intersections of elements $K_h^{+},K_h^{-}\in \mathcal{T}_{h}$ (i.e., the edges). We define the conormal $n_h^{+}$ on such an intersection $e_{h} \in \mathcal{E}_{h}$ of elements $K_h^{+}$ and $K_h^{-}$ by demanding that \\[1mm]
$\bullet$ $n_h^{+}$ is a unit vector,\\
$\bullet$ $n_h^{+}$ is tangential to (the planar triangle) $K_h^{+}$, \\
$\bullet$ in each point $x \in e_h$ we have that $n_h^{+} \cdot (y-x) \leq 0$ for all $y \in K_h^{+}$. \\[1mm]
Analogously one can define the conormal $n_h^{-}$ on $e_h$ by exchanging $K_h^{+}$ with $K_h^{-}$. A discrete DG space associated with $\Gamma_{h}$ is given by
\[V_{h} := \{v_{h} \in L^{2}(\Gamma_{h})\ : \ \left. v_{h} \right |_{K_{h}} \in P^{1}(K_{h})\ \forall K_{h}\in \mathcal{T}_{h}\}\]
i.e. the space of piecewise linear functions which are globally in $L^{2}(\Gamma_{h})$. 
For $v_h\in V_{h}$, let 
\[
v_h^{+/-}:= v_h \big{|}_{\partial K_h^{+/-}}.
\]
We can now define a discrete DG formulation on $\Gamma_{h}$ for a given function $f_{h} \in L^{2}(\Gamma_{h})$ (note that, in general, this is not a finite element function, it will be related to the function $f$ given in problem $(\mathbf{P}_{\Gamma})$ later on, see (\ref{eq:rel_f_fh}) below):
\\
\\
$\left(\mathbf{P}_{\Gamma_{h}}^{IP}\right)$ Find $u_{h} \in V_{h}$  such that
\begin{equation} \label{eq:InteriorPenaltyGammah}
a_{\Gamma_{h}}^{IP}(u_{h},v_{h}) = \sum_{K_{h} \in \mathcal{T}_{h}}\int_{K_{h}}f_{h} v_{h}\ dA_{h}\ \forall v_{h} \in V_{h}
\end{equation}
where
\begin{align}\label{eq:InteriorPenaltyGammahForm}
a_{\Gamma_{h}}^{IP}(u_{h},v_{h}&) := \sum_{K_{h} \in \mathcal{T}_{h}}\int_{K_{h}}\nabla_{\Gamma_{h}}u_{h}\cdot \nabla_{\Gamma_{h}}v_{h} + u_{h} v_{h}\ dA_{h} \notag \\ 
&- \sum_{e_{h} \in \mathcal{E}_{h}}\int_{e_{h}}(u_{h}^{+}-u_{h}^{-})\frac{1}{2}(\nabla_{\Gamma_{h}}v_{h}^{+}\cdot n_{h}^{+} - \nabla_{\Gamma_{h}}v_{h}^{-}\cdot n_{h}^{-}) + (v_{h}^{+}-v_{h}^{-})\frac{1}{2}(\nabla_{\Gamma_{h}}u_{h}^{+}\cdot n_{h}^{+} - \nabla_{\Gamma_{h}}u_{h}^{-}\cdot n_{h}^{-})\ ds_{h} \notag \\
&+ \sum_{e_{h} \in \mathcal{E}_{h}}\int_{e_{h}}\beta_{e_{h}}(u_{h}^{+}-u_{h}^{-})(v_{h}^{+}-v_{h}^{-})\ ds_{h}.
\end{align}
The penalty parameters
$\beta_{e_{h}}$ are given by $\beta_{e_{h}} = \omega_{e_{h}}h_{e_{h}}^{-1}$ where $h_{e_{h}}$ is some length scale
associated with the intersection $e_{h}$ (for instance, the edge length). The interior penalty parameters $\omega_{e_{h}}$ are uniformly bounded with respect to $h := \max_{e_{h} \in \mathcal{E}_{h}} h_{e_{h}}$. 
\begin{remark}
This formulation corresponds to the one found in \cite{arnold2002unified} in the case when the domain is flat and is similar in nature to the original formulation of the IP method found in \cite{arnold1982interior} for which the conormals $n_{h}^{+/-}$ are associated with their respective gradient terms rather than the scalar terms. It is important to point out that this formulation is not equivalent to using the formulation found in \cite{arnold2002unified} on $\Gamma_{h}$. We will discuss this issue further in Section~\ref{sec:numerics}. 
\end{remark}

We now define a norm on the space of piecewise smooth functions: 
\begin{definition}\label{def:DGNorm}
For $u_{h} \in V_{h}$ we define
\[ |u_{h}|_{1,h}^{2} := \sum_{K_{h} \in \mathcal{T}_{h}} \norm{u_{h}}_{H^{1}(K_{h})}^{2}\    \ , \    \ |u|_{*,h}^{2} := \sum_{e_{h} \in \mathcal{E}_{h}} h_{e_{h}}^{-1}  \norm{u_{h}^{+} - u_{h}^{-}}_{L^{2}(e_{h})}^{2}.\]
The DG norm is given by
\[ \norm{u_{h}}_{DG}^{2} := |u_{h}|_{1,h}^{2} + |u_{h}|_{*,h}^{2}. \]
\end{definition}

\begin{lemma}
Let $\mathcal{E}_{K_{h}}$ denote the set containing the individual edges of element $K_{h}$. Then if $\beta_{e_{h}} = \omega_{e_{h}}h_{e_{h}}^{-1}$ with
\begin{align}\label{penaltyParameterBound} 
\omega_{e_{h}} > \max_{\substack{K_{h} \in \mathcal{T}_{h}: \\ e_{h} \subset \partial K_{h}}}
  \frac{1}{2}\sum_{\tilde{e_{h}} \in \mathcal{E}_{K_{h}}}\frac{|\tilde{e_{h}}|^{2}}{|K_{h}|}\ \mbox{for all} \ \ e_{h} \in \mathcal{E}_{h},
\end{align}
then $a_{\Gamma_h}^{IP}$ is stable and bounded. Hence there is a unique solution $u_{h} \in V_{h}$ of $(\mathbf{P}_{\Gamma_{h}}^{IP})$ which satisfies
\begin{equation}\label{eq:stabilityEstimateGamma1}
\norm{u_{h}}_{DG} \leq C \norm{f_{h}}_{L^{2}(\Gamma_{h})}~.
\end{equation}
\end{lemma}

\begin{proof}
Boundedness and stability of (\ref{eq:InteriorPenaltyGammahForm}) follow in a similar way as for the classical IP method (see \cite{arnold2002unified} for more details) since all the arguments apply to $\Gamma_{h}$. For the lower bound of the penalty parameters, the proof of Lemma 2.1 in \cite{ainsworth2011constant} applies straightforwardly to the surface $\Gamma_{h}$. Note that the reason why these results naturally extend onto $\Gamma_{h}$ is that the latter is composed of planar triangles. By Lax-Milgram, the uniqueness property follows.
\end{proof}

Our goal now is to compare the solution $u \in H^2(\Gamma)$ of $(\mathbf{P}_{\Gamma})$ with the solution $u_{h} \in V_{h}$ of $(\mathbf{P}_{\Gamma_{h}}^{IP})$ but these functions are defined on different domains. The approach suggested in \cite{dziuk1988finite} is to lift functions defined on the discrete surface $\Gamma_{h}$ onto the smooth surface $\Gamma$. 
\begin{definition}
For any function $w$ defined on $\Gamma_{h}$ we define the \emph{lift} onto $\Gamma$ by
\[ w^{l}(\xi) := w(x(\xi)),\ \xi \in \Gamma, \]
where by (\ref{eq:uniquePoint}) and the non-overlapping of the triangular elements, $x(\xi)$ is defined as the unique solution of
\[ x = \xi + d(x)\nu(\xi).\]
\end{definition}
Extending $w^{l}$ constantly along the lines $s \mapsto \xi + s\nu(\xi)$ we obtain a function defined on $U$. In particular, we 
\begin{equation} \label{eq:rel_f_fh}
\mbox{define }f_{h} \mbox{ such that } f_{h}^{l} = f \mbox{ on } \Gamma.
\end{equation}
By (\ref{eq:uniquePoint}), for every $K_{h} \in \mathcal{T}_{h}$, there is a unique curved triangle $K_{h}^{l} := \xi(K_{h}) \subset \Gamma$. Note that we assumed $\xi(x)$ is a bijection so multiple coverings are not permitted. We now define the regular, conforming triangulation $\mathcal{T}_{h}^{l}$ of $\Gamma$ such that
\[ \Gamma = \bigcup_{K_{h}^{l} \in \mathcal{T}_{h}^{l}} K_{h}^{l}. \] 
The triangulation $\mathcal{T}_{h}^{l}$ of $\Gamma$ is thus induced by the triangulation $\mathcal{T}_{h}$ of $\Gamma_{h}$ via the lift. Similarly, $e_{h}^{l}:=\xi(e_{h}) \in \mathcal{E}_{h}^{l}$ are the unique curved edges.

The appropriate function space for lifted functions is given by
\[V_{h}^{l} := \{v_{h}^{l} \in L^{2}(\Gamma)\ : \ v_{h}^{l}(\xi) = v_{h}(x(\xi))\ \mbox{with some}\ v_{h} \in V_{h}\}.\]
Note that the DG norm for functions $u_{h}^{l} \in V_{h}^{l}$ is the same one as in Definition \ref{def:DGNorm} but with the triangulation $\mathcal{T}_{h}^{l}$ instead and corresponding length scale $h_{e_{h}^{l}}$ associated with $e_{h}^{l}$. The context of its use makes it clear which DG norm we are dealing with. Furthermore, we observe that 
\begin{equation} \label{eq:ineq_length_edges}
h_{e_{h}}^{l} \geq h_{e_h}
\end{equation}
since the deformation of the straight edges can only increase their length. We now prove some geometric error estimates relating $\Gamma$ to $\Gamma_{h}$.
\begin{lemma} \label{Gamma2GammahSmall}
Let $\Gamma$ be a compact smooth connected and oriented surface in $\mathbb{R}^{3}$ and $\Gamma_{h}$ its linear interpolation with outward unit normal $\nu_{h}$. Let $H = \nabla^{2}d$ and $P_{h} = I - \nu_{h} \otimes \nu_{h}$. Furthermore, we denote by $\delta_{h}$ the local area deformation when transforming $K_{h}$ to $K_{h}^{l}$ i.e. $\delta_{h}dA_{h} = dA$ and $\delta_{e_{h}}$ the local edge deformation when transforming $e_{h}$ to $e_{h}^{l}$ i.e. $\delta_{e_{h}}ds_{h} = ds$. Then we have
\begin{align*}
\norm{d}&_{L^{\infty}(\Gamma)} \leq Ch^{2},\ \norm{1-\delta_{h}}_{L^{\infty}(\Gamma)} \leq Ch^{2},\ \norm{\nu-\nu_{h}}_{L^{\infty}(\Gamma)} \leq Ch,\ \norm{P-R_{h}}_{L^{\infty}(\Gamma)} \leq Ch^{2}\\ &\norm{1-\delta_{e_{h}}}_{L^{\infty}(\Gamma)} \leq Ch^{2},\ \norm{P-R_{e_{h}}}_{L^{\infty}(\Gamma)} \leq Ch^{2}\ \mbox{and}\ \norm{n-P n_{h}^{l}}_{L^{\infty}(e_{h}^{l})} \leq C h^{2}
\end{align*}
where $R_{h} :=  \frac{1}{\delta_{h}}P(I-d H)P_{h}(I-d H)$ and $R_{e_{h}} :=  \frac{1}{\delta_{e_{h}}}P(I-d H)P_{h}(I-d H)$.
\end{lemma}
\begin{proof}
All these geometric estimates follow from standard interpolation theory via the linear interpolation of $\Gamma$. Proofs of the first four estimates can be found in \cite{dziuk1988finite}, the fifth (and thus sixth) one in \cite{GieMue_prep}. The last estimate is a corollary of another result in \cite{GieMue_prep}, which states that if $n$ and $n_{h}^{l}$ are given as before and $\tau$ denotes a unit tangent vector on some $e_{h}^{l} \in \mathcal{E}_{h}^{l}$, we have
\[|(\tau , n_{h}^{l})| \leq C h^{2}, |1 - (n , n_{h}^{l})| \leq C h^{2}.\]
Writing $Pn_{h}^{l} = (\tau,Pn_{h}^{l})\tau + (n,Pn_{h}^{l})n$, we deduce that indeed
\begin{align*}
\norm{n-P n_{h}^{l}}_{L^{\infty}(e_{h}^{l})} &= \norm{n - (\tau,Pn_{h}^{l})\tau - (n,Pn_{h}^{l})n}_{L^{\infty}(e_{h}^{l})} \\ 
&\leq |1 - (n,Pn_{h}^{l})| + |(\tau,Pn_{h}^{l})| = |1 - (n,n_{h}^{l})| + |(\tau,n_{h}^{l})| = O(h^{2}).
\end{align*}
\end{proof} 

\begin{lemma}
Let $u_{h} \in V_{h}$ satisfy (\ref{eq:stabilityEstimateGamma1}). Then $u_{h}^{l} \in V_{h}^{l}$ satisfies
\begin{equation}\label{eq:stabilityEstimateGamma2}
\norm{u_{h}^{l}}_{DG} \leq C \norm{f}_{L^{2}(\Gamma)}
\end{equation}
for sufficiently small $h$.
\end{lemma}
\begin{proof}
We first show that that for any function $v_h\in V_h$,  
\begin{align}\label{eq:normestimate}
\norm{v_h}_{DG} \geq C\norm{v_h^l}_{DG}.
\end{align}
The $|\cdot|_{1,h}^{2}$ component of the DG norm is dealt with in exactly the same way as in \cite{dziuk1988finite}. Similarly, making use of Lemma \ref{Gamma2GammahSmall} and the fact that $h_{e_{h}}^{-1} \geq h_{e_{h}^{l}}^{-1}$ (see (\ref{eq:ineq_length_edges})), we obtain the following for the $|\cdot|_{*,h}^{2}$ component of the DG norm:
\begin{align*} 
&\sum_{e_{h} \in \mathcal{E}_{h}}h_{e_{h}}^{-1}\int_{e_{h}} (v_{h}^{+} - v_{h}^{-})^{2}\ ds_{h} \geq \sum_{e_{h}^{l} \in \mathcal{E}_{h}^{l}}h_{e_{h}^{l}}^{-1}\int_{e_{h}^{l}} (v_{h}^{l+} - v_{h}^{l-})^{2}\frac{1}{\delta_{e_{h}}}\ ds \\
&= \sum_{e_{h}^{l} \in \mathcal{E}_{h}^{l}}h_{e_{h}^{l}}^{-1}\int_{e_{h}^{l}} (v_{h}^{l+} - v_{h}^{l-})^{2}\ ds + \sum_{e_{h}^{l} \in \mathcal{E}_{h}^{l}}h_{e_{h}^{l}}^{-1}\int_{e_{h}^{l}} (v_{h}^{l+} - v_{h}^{l-})^{2}\left(\frac{1}{\delta_{e_{h}}}-1\right)\ ds\\ 
&\geq \left(1-Ch^2 \right)\sum_{e_{h}^{l} \in \mathcal{E}_{h}^{l}}h_{e_{h}^{l}}^{-1}\int_{e_{h}^{l}} (v_{h}^{l+} - v_{h}^{l-})^{2}\ ds
\end{align*}  
which yields the desired estimate for sufficiently small $h$. Noting that $\norm{f_h}_{L^2(\Gamma_h)} \leq C\norm{f_h^l}_{L^2(\Gamma)} = C\norm{f}_{L^2(\Gamma)}$ (see \cite{dziuk1988finite}), we can extend the stability estimate (\ref{eq:stabilityEstimateGamma1}) to the lifted discrete function $u_h^l$ as required.  
\end{proof}

We now define a bilinear form on $\Gamma$ induced by $a_{\Gamma_h}^{IP}$ and the lifting operator:
\begin{align}\label{eq:InteriorPenaltyGammaForm}
a_{\Gamma}^{IP}(u,v) &:= 
  \sum_{K_{h}^{l} \in \mathcal{T}_{h}^{l}}
    \int_{K_{h}^{l}}\nabla_{\Gamma}u\cdot \nabla_{\Gamma}v + u v\ dA \notag \\ 
&- \sum_{e_h^l \in \mathcal{E}^l_h}
    \int_{e_{h}^{l}}(u^{+}-u^{-})\frac{1}{2}(\nabla_{\Gamma}v^{+}\cdot n^{+} - \nabla_{\Gamma}v^{-}\cdot n^{-}) + 
                     (v^{+}-v^{-})\frac{1}{2}(\nabla_{\Gamma}u^{+}\cdot n^{+} - \nabla_{\Gamma}u^{-}\cdot n^{-})\ ds 
  \notag \\
&+ \sum_{e^l_{h} \in \mathcal{E}^l_{h}}
    \int_{e_{h}^{l}}\beta_{e^l_{h}}(u^{+}-u^{-})(v^{+}-v^{-})\ ds
\end{align}
where $n^{+}$ and $n^{-}$ are respectively the unit surface conormals to $K_h^{l+}$ and $K_h^{l-}$ on $e_{h}^{l} \in \mathcal{E}_{h}^{l}$ and the penalty parameters are defined to be $\beta_{e_h^l} := \frac{\beta_{e_h}}{\delta_{e_h}}$.  
This bilinear form is well defined for functions $u,v \in H^2(\Gamma)+V_{h}^{l}$, and since the weak solution $u$ given by \eqref{eq:weakH1} is in $H^{2}(\Gamma)$ it satisfies
\begin{equation} \label{eq:InteriorPenaltyGamma}
a_{\Gamma}^{IP}(u,v) = \sum_{K_{h}^{l} \in \mathcal{T}_{h}^{l}}\int_{K_{h}^{l}}f v\ dA\ \ \forall v \in H^{2}(\Gamma) + V_{h}^{l}.
\end{equation} 
We now state and prove a technical estimate of importance for boundedness and stability of $a_{\Gamma}^{IP}$. 
\begin{lemma}\label{thm:traceEstimatesLift}
Let $w \in H^2(\Gamma)$ and $w_{h}^{l} \in V_{h}^{l}$. Then for sufficiently small $h$,
\begin{equation}\label{eq:traceEstimateLift}
\norm{\nabla_{\Gamma}(w + w_{h}^{l})}_{L^{2}(\partial K_{h}^{l})}^{2} \leq C \left( \frac{1}{h} \norm{\nabla_{\Gamma}(w + w_{h}^{l})}_{L^{2}(K_{h}^{l})}^{2} + h \norm{w}_{H^{2}(K_{h}^{l})}^{2} \right). 
\end{equation}
\end{lemma}
\begin{proof}
We define $\tilde{w}$ and $w_{h}$ such that their lifts coincide with $w$ and $w_{h}^{l}$, respectively. Since $\tilde{w} + w_{h} \in H^2(K_{h})$ on each $K_{h}$, applying the trace theorem and a standard scaling argument on $K_{h} \in \mathcal{T}_{h}$ yields
\[\int_{\partial K_{h}}|\nabla_{\Gamma_{h}}(\tilde{w} + w_{h})|^2\ ds_{h} \leq C \left( \frac{1}{h} \int_{K_{h}}|\nabla_{\Gamma_{h}}(\tilde{w}+w_{h})|^{2}\ dA_{h} + h \int_{K_{h}} | \nabla_{\Gamma_{h}}^{2} \tilde{w}|^{2}\ dA_{h} \right) \]
where we used that $\nabla_{\Gamma_{h}}^{2} w_{h} = 0$ thanks to the linearity of the finite element functions.
Lifting the estimate onto $\Gamma$ and making use of estimate (2.17) in \cite{demlow2009higher}, we have
\begin{align*}
&\int_{\partial K_{h}^{l}}\nabla_{\Gamma}(w + w_{h}^{l}) \cdot R_{e_{h}} \nabla_{\Gamma}(w + w_{h}^{l})\ ds \leq C \left( \frac{1}{h} \int_{K_{h}^{l}}|\nabla_{\Gamma}(w + w_{h}^{l})|^2\ dA + h \int_{K_{h}^{l}}  |\nabla_{\Gamma}^{2} w|^2 + |\nabla_{\Gamma} w|^2\ dA \right)
\end{align*}
where $R_{e_{h}}$ is given as in Lemma \ref{Gamma2GammahSmall}. We thus obtain
\[\left(1 - Ch^{2} \right)\int_{\partial K_{h}^{l}}|\nabla_{\Gamma}(w + w_{h}^{l})|^2\ ds \leq C \left( \frac{1}{h} \int_{K_{h}^{l}}|\nabla_{\Gamma}(w + w_{h}^{l})|^{2}\ dA + h \int_{K_{h}^{l}} | \nabla_{\Gamma}^{2} w |^2 + |\nabla_{\Gamma} w|^2\ dA \right) \]
which yields the desired inequality for $h$ small enough. 
\end{proof}

Following the lines of \cite{arnold2002unified} along with Lemma \ref{thm:traceEstimatesLift}, we can show the boundedness estimate
\begin{equation}\label{eq:boundednessGammaIP}
a_{\Gamma}^{IP}(w + w_{h}^{l},v_{h}^{l}) \leq C_{b}^{l}\left( \norm{w + w_{h}^{l}}_{DG} + h^2 \norm{w}_{H^{2}(\Gamma)}\right) \norm{v_{h}^{l}}_{DG}   \ \ \mbox{for all} \  \ w \in H^2(\Gamma), \ w_{h}^{l}, v_{h}^{l} \in V_h^l, 
\end{equation} 
and the stability estimate
\begin{equation}\label{eq:stabilityGammaIP}
a_{\Gamma}^{IP}(w_{h}^{l},w_{h}^{l}) \geq C_{s}^{l} \norm{w_{h}^{l}}_{DG}^{2} \  \  \mbox{for all} \  \  w_{h}^{l} \in V_{h}^{l}.   
\end{equation}

\section{Convergence}
\label{sec:Convergence}
\begin{theorem}\label{aprioriErrorEstimateIP}
Let $u \in H^2(\Gamma)$ and $u_{h} \in V_{h}$ denote the solutions to $(\mathbf{P}_{\Gamma})$ and $(\mathbf{P}_{\Gamma_{h}}^{IP})$, respectively. Denote by $u_{h}^{l} \in V_{h}^{l}$ the lift of $u_{h}$ onto $\Gamma$. Then
\[ \norm{u-u_{h}^{l}}_{L^{2}(\Gamma)} + h\norm{u-u_{h}^{l}}_{DG} \leq Ch^2\norm{f}_{L^{2}(\Gamma)}.\]
\end{theorem}
The proof will follow an argument similar to the one outlined in \cite{arnold2002unified}. Using the stability result (\ref{eq:stabilityGammaIP}), we have
\begin{equation} \label{eq:first_sec4}
\norm{\phi_{h}^{l}- u_{h}^{l}}_{DG}^{2} \leq 
  \frac{1}{C^l_{s}} a_{\Gamma}^{IP}(\phi_{h}^{l}-u_{h}^{l},\phi_{h}^{l}- u_{h}^{l}) = \frac{1}{C^l_{s}} a_{\Gamma}^{IP}(u-u_{h}^{l},\phi_{h}^{l}-u_{h}^{l}) + \frac{1}{C^l_{s}} a_{\Gamma}^{IP}(\phi_{h}^{l}-u,\phi_{h}^{l}-u_{h}^{l})
\end{equation}
where $\phi_{h}^{l} \in V_{h}^{l}$. Since we do not directly have Galerkin orthogonality the first term is not zero, and the second term will require an interpolation estimate. These terms are addressed by the following lemmas:
\begin{lemma}\label{interpolationEstimate}
For a given $w \in H^{2}(\Gamma)$ there exists an interpolant $I_{h}^{l}w \in V_{h}^{l}$ such that
\[\norm{w - I_{h}^{l}w}_{L^{2}(\Gamma)} + h\norm{\nabla_{\Gamma}(w - I_{h}^{l}w)}_{L^{2}(\Gamma)} \leq Ch^{2}\left(\norm{\nabla^{2}_{\Gamma}w}_{L^{2}(\Gamma)} + h\norm{\nabla_{\Gamma}w}_{L^{2}(\Gamma)}\right). \]
\end{lemma}
\begin{proof}
See \cite{dziuk1988finite}.
\end{proof}
\begin{lemma} \label{PerturbedGalerkinOrthogonality}
Let $u$ and $u_{h}^{l}$ be given as in Theorem \ref{aprioriErrorEstimateIP} and define
the functional $E_h$ on $V_{h}^{l}$ by
\begin{equation*} 
E_h(v_h^l) := a_{\Gamma}^{IP}(u-u^{l}_{h},v^{l}_{h}).
\end{equation*}

Then $E_h$ can be written as
\begin{align*}
E_{h}(v_{h}^{l}) &= \sum_{K_{h}^{l} \in \mathcal{T}_{h}^{l}}\int_{K_{h}^{l}}(R_{h}-P)\nabla_{\Gamma}u_{h}^{l} \cdot\nabla_{\Gamma}v_{h}^{l}+\left(\frac{1}{\delta_{h}}-1\right)u_{h}^{l}v_{h}^{l} + \left(1-\frac{1}{\delta_{h}}\right)fv_{h}^{l}\ dA \\
&+ \sum_{e_{h}^{l} \in \mathcal{E}_{h}^{l}}\int_{e_{h}^{l}}(u_{h}^{l+}-u_{h}^{l-})\frac{1}{2}(\nabla_{\Gamma}v_{h}^{l+}\cdot n^{+} - \nabla_{\Gamma} v_{h}^{l-}\cdot n^{-}) \notag \\
& \hspace{1.5cm} -(u_{h}^{l+}-u_{h}^{l-})\frac{1}{2}(P_{h}^{+}(I-dH)P\nabla_{\Gamma}v_{h}^{l+}\cdot n_{h}^{l+} - P_{h}^{-}(I-dH)P\nabla_{\Gamma}v_{h}^{l-}\cdot n_{h}^{l-})\frac{1}{\delta_{e_{h}}}\ ds \notag \\ 
&+ \sum_{e_{h}^{l} \in \mathcal{E}_{h}^{l}}\int_{e_{h}^{l}}(v_{h}^{l+}-v_{h}^{l-})\frac{1}{2}(\nabla_{\Gamma}u_{h}^{l+}\cdot n^{+} - \nabla_{\Gamma} u_{h}^{l-}\cdot n^{-}) \notag \\
& \hspace{1.5cm} -(v_{h}^{l+}-v_{h}^{l-})\frac{1}{2}(P_{h}^{+}(I-dH)P\nabla_{\Gamma}u_{h}^{l+}\cdot n_{h}^{l+} - P_{h}^{-}(I-dH)P\nabla_{\Gamma}u_{h}^{l-}\cdot n_{h}^{l-})\frac{1}{\delta_{e_{h}}}\ ds 
\end{align*}
where $R_h$ is given as in Lemma \ref{Gamma2GammahSmall}. Furthermore, $E_h$ scales quadratically in $h$ i.e.
\begin{equation} \label{eq:Eh_quad}
|E_h(v_{h}^{l})| \leq Ch^2\norm{f}_{L^{2}(\Gamma)}\norm{v_{h}^{l}}_{DG}.
\end{equation}
\end{lemma}
\begin{remark}
Note that the error functional $E_{h}$ in Lemma \ref{PerturbedGalerkinOrthogonality} includes all of the terms of the classical FEM setting (see \cite{dziuk1988finite}) as well as additional terms arising from the jumps across elements which characterise the DG method. 
\end{remark}
The proof of Lemma \ref{PerturbedGalerkinOrthogonality} will be the main part of this section. Before we give its full proof, we will complete that of Theorem~\ref{aprioriErrorEstimateIP} assuming this result. Using the estimate (\ref{eq:first_sec4}) given at the start of the proof of Theorem \ref{aprioriErrorEstimateIP}, the boundedness result (\ref{eq:boundednessGammaIP}), 
the elliptic regularity result (\ref{eq:EllipticRegularity}) and the quadratic scaling of $E_h$ in $h$ (\ref{eq:Eh_quad}), we have
\begin{align*}
\norm{\phi_{h}^{l}- u_{h}^{l}}_{DG}^{2}  &\leq 
     \frac{1}{C^l_{s}} E_{h}(\phi_{h}^{l}-u_{h}^{l}) + 
     \frac{1}{C^l_{s}}
     a_{\Gamma}^{IP}(\phi_{h}^{l}-u,\phi_{h}^{l}-u_{h}^{l})\\  
&\leq \frac{1}{C^l_{s}} E_{h}(\phi_{h}^{l}-u_{h}^{l}) +
       \frac{C_b^l}{C^l_{s}}
       \left( \norm{\phi_{h}^{l}-u}_{DG} + h^2 \norm{u}_{H^{2}(\Gamma)} \right)\norm{\phi_{h}^{l}-u_{h}^{l}}_{DG} \notag \\
  &\leq Ch^2\norm{f}_{L^{2}(\Gamma)}\norm{\phi_{h}^{l}-u_{h}^{l}}_{DG} +
  C\left(\norm{\phi_{h}^{l}-u}_{DG} + h^2 \norm{f}_{L^{2}(\Gamma)} \right)\norm{\phi_{h}^{l}-u_{h}^{l}}_{DG} \notag, 
\end{align*} 
thus
\[ \norm{\phi_{h}^{l}- u_{h}^{l}}_{DG} \leq Ch^2\norm{f}_{L^{2}(\Gamma)} + C\norm{\phi_{h}^{l}-u}_{DG}. \]
Now taking the continuous interpolant $\phi_{h}^{l} = I^l_hu$ and using
Lemma~\ref{interpolationEstimate} we obtain
\begin{align*}
  \norm{u-u_h^l}_{DG} &\leq \norm{u-\phi_{h}^{l}}_{DG} + \norm{\phi_{h}^{l}- u_{h}^{l}}_{DG}
      \leq \norm{u-\phi_{h}^{l}}_{DG} + Ch^2\norm{f}_{L^{2}(\Gamma)} +C\norm{\phi_{h}^{l}-u}_{DG} \leq Ch\norm{f}_{L^{2}(\Gamma)}
\end{align*}
as required. The $L^{2}$ error estimate can be derived using the usual Aubin-Nitsche trick in a similar way as in \cite{dziuk1988finite}, which concludes the proof of Theorem \ref{aprioriErrorEstimateIP}. 

\begin{proof}[Proof of Lemma \ref{PerturbedGalerkinOrthogonality}]
The expression for the error functional $E_{h}$ given in Lemma \ref{PerturbedGalerkinOrthogonality} is obtained by considering the difference between the two equations (\ref{eq:InteriorPenaltyGamma}) and (\ref{eq:InteriorPenaltyGammah}). In order to do this, the integrals of (\ref{eq:InteriorPenaltyGammah}) have to first be lifted onto $\Gamma$. For every $K_{h} \in \mathcal{T}_{h}$, we have 
\begin{align*} 
\int_{K_{h}} \nabla_{\Gamma_{h}}u_{h} \cdot \nabla_{\Gamma_{h}}v_{h} + u_{h}v_{h}\ dA_{h} = \int_{K_{h}^{l}} R_{h}\nabla_{\Gamma}u_{h}^{l} \cdot\nabla_{\Gamma}v_{h}^{l} + \frac{1}{\delta_{h}}u_{h}^{l} v_{h}^{l}\ dA.
\end{align*}
Furthermore, for every $e_{h} \in \mathcal{E}_{h}$, we have
\begin{align*}
\int_{e_{h}}&(u_{h}^{+}-u_{h}^{-})\frac{1}{2}(\nabla_{\Gamma_{h}}v_{h}^{+}\cdot n_{h}^{+} - \nabla_{\Gamma_{h}}v_{h}^{-}\cdot n_{h}^{-}) + (v_{h}^{+}-v_{h}^{-})\frac{1}{2}(\nabla_{\Gamma_{h}}u_{h}^{+}\cdot n_{h}^{+} - \nabla_{\Gamma_{h}}u_{h}^{-}\cdot n_{h}^{-})\ ds_{h} \notag \\
= \int_{e_{h}^{l}}&(u_{h}^{l+}-u_{h}^{l-})\frac{1}{2}(P_{h}^{+}(I-dH)P\nabla_{\Gamma}v_{h}^{l+}\cdot n_{h}^{l+} - P_{h}^{-}(I-dH)P\nabla_{\Gamma}v_{h}^{l-}\cdot n_{h}^{l-}) \notag \\
&+ (v_{h}^{l+}-v_{h}^{l-})\frac{1}{2}(P_{h}^{+}(I-dH)P\nabla_{\Gamma}u_{h}^{l+}\cdot n_{h}^{l+} - P_{h}^{-}(I-dH)P\nabla_{\Gamma}u_{h}^{l-}\cdot n_{h}^{l-})\frac{1}{\delta_{e_{h}}}\ ds.  
\end{align*}
And finally, we have using $\beta_{e^l_{h}} = \frac{\beta_{e_{h}}}{\delta_{e_{h}}}$ that
\begin{align*}
\int_{e_{h}}\beta_{e_{h}}(u_{h}^{+}-u_{h}^{-})(v_{h}^{+}-v_{h}^{-})\ ds_{h} = 
    \int_{e_{h}^{l}}\beta_{e^l_{h}}(u_{h}^{l+}-u_{h}^{l-})(v_{h}^{l+}-v_{h}^{l-})\,ds.
\end{align*}
The right-hand side of (\ref{eq:InteriorPenaltyGammah}) gets transformed in a similar way:
\[\sum_{K_{h} \in \mathcal{T}_{h}}\int_{K_{h}}f_{h} v_{h}\ dA_{h} = \sum_{K_{h}^{l} \in \mathcal{T}_{h}^{l}}\int_{K_{h}^{l}}f v_{h}^{l}\frac{1}{\delta_{h}}\ dA .\]
Making use of the above, the difference between the two equations (\ref{eq:InteriorPenaltyGamma}) and (\ref{eq:InteriorPenaltyGammah}) yields
\begin{align*}
0 &= a_{\Gamma}^{IP}(u,v_{h}^{l}) - \sum_{K_{h}^{l} \in \mathcal{T}_{h}^{l}}\int_{K_{h}^{l}}f v_{h}^{l}\ dA - a_{\Gamma_{h}}^{IP}(u_{h},v_{h}) + \sum_{K_{h} \in \mathcal{T}_{h}}\int_{K_{h}}f_{h} v_{h}\ dA_{h} \\
&= a_{\Gamma}^{IP}(u-u_{h}^{l},v_{h}^{l}) - E_{h}^{l}(v_{h}^{l}) 
\end{align*}
as required.

Finally we need to show that the error functional $E_h$ scales quadratically in $h$ i.e. 
\[|E_h(v_{h}^{l})| \leq Ch^2\norm{f}_{L^{2}(\Gamma)}\norm{v_{h}^{l}}_{DG}.\] 
To this end
we need to show that the additional terms arising in the error functional $E_{h}$ do not affect the convergence rates expressed in \cite{dziuk1988finite}.
The first term of the error functional $E_{h}$ (the element integral) is the one resulting from the standard surface FEM approach. By Lemma \ref{Gamma2GammahSmall} this term scales quadratically in $h$ and making use of the stability estimate (\ref{eq:stabilityEstimateGamma2}) this term scales like the right-hand side of (\ref{eq:Eh_quad}). We will now get a bound for the third term of $E_{h}$, for which 
we have the following:
\begin{align*}
\sum_{e_{h}^{l} \in \mathcal{E}_{h}}&\int_{e_{h}^{l}}(v_{h}^{l+}-v_{h}^{l-})\frac{1}{2}(\nabla_{\Gamma}u_{h}^{l+}\cdot n^{+} - \nabla_{\Gamma} u_{h}^{l-}\cdot n^{-})\left(1+\frac{1}{\delta_{e_{h}}}-\frac{1}{\delta_{e_{h}}}\right) \\
&-(v_{h}^{l+}-v_{h}^{l-})\frac{1}{2}(P_{h}^{+}(I-dH)P\nabla_{\Gamma}u_{h}^{l+}\cdot n_{h}^{l+} - P_{h}^{-}(I-dH)P\nabla_{\Gamma}u_{h}^{l-}\cdot n_{h}^{l-})\frac{1}{\delta_{e_{h}}}\ ds \\
= \sum_{e_{h}^{l} \in \mathcal{E}_{h}^{l}}&\int_{e_{h}^{l}} (v_{h}^{l+}-v_{h}^{l-})\frac{1}{2}(\nabla_{\Gamma}u_{h}^{l+}\cdot n^{+} - \nabla_{\Gamma} u_{h}^{l-}\cdot n^{-})\left(1-\frac{1}{\delta_{e_{h}}}\right) + \frac{1}{\delta_{e_{h}}}(v_{h}^{l+}-v_{h}^{l-})\frac{1}{2}\Big((\nabla_{\Gamma}u_{h}^{l+}\cdot n^{+}\\ 
&- \nabla_{\Gamma}u_{h}^{l-}\cdot n^{-}) -(P_{h}^{+}(I-dH)P\nabla_{\Gamma}u_{h}^{l+}\cdot n_{h}^{l+}
- P_{h}^{-}(I-dH)P\nabla_{\Gamma}u_{h}^{l-}\cdot n_{h}^{l-}) \Big)\ ds.
\end{align*}
Making use of standard arguments as found in \cite{arnold2002unified} along with Lemma \ref{thm:traceEstimatesLift}, Lemma \ref{Gamma2GammahSmall} and the stability estimate (\ref{eq:stabilityEstimateGamma2}) it is clear that the first component in the above scales appropriately, so all we have to deal with is the second component. We first note that
\begin{align*} \nabla_{\Gamma}u_{h}^{l+}\cdot n^{+} - P_{h}^{+}(I-dH)P\nabla_{\Gamma}u_{h}^{l+}\cdot n_{h}^{l+} = \nabla_{\Gamma}u_{h}^{l+}\cdot n^{+} - \nabla_{\Gamma}u_{h}^{l+}\cdot P(I-dH)P_{h}^{+}n_{h}^{l+} \\
= \nabla_{\Gamma}u_{h}^{l+}\cdot n^{+} - \nabla_{\Gamma}u_{h}^{l+}\cdot P(I-dH)n_{h}^{l+} = \nabla_{\Gamma}u_{h}^{l+}\cdot (n^{+} - Pn_{h}^{l+}) + d H\nabla_{\Gamma}u_{h}^{l+}\cdot n_{h}^{l+},
\end{align*}
hence
\begin{align*}
\sum_{e_{h}^{l} \in \mathcal{E}_{h}^{l}}&\int_{e_{h}^{l}}\frac{1}{\delta_{e_{h}}}(v_{h}^{l+}-v_{h}^{l-})\frac{1}{2}\Big((\nabla_{\Gamma}u_{h}^{l+}\cdot n^{+} - \nabla_{\Gamma}u_{h}^{l-}\cdot n^{-})\\ 
&-(P_{h}^{+}(I-dH)P\nabla_{\Gamma}u_{h}^{l+}\cdot n_{h}^{l+} - P_{h}^{-}(I-dH)P\nabla_{\Gamma}u_{h}^{l-}\cdot n_{h}^{l-}) \Big)\ ds \\
= \sum_{e_{h}^{l} \in \mathcal{E}_{h}^{l}}&\int_{e_{h}^{l}}\frac{1}{\delta_{e_{h}}}(v_{h}^{l+}-v_{h}^{l-})\frac{1}{2}\Big((n^{+}-Pn_{h}^{l+})\cdot \nabla_{\Gamma}u_{h}^{l+} + d H (\nabla_{\Gamma}u_{h}^{l+}\cdot n_{h}^{l+}-\nabla_{\Gamma}u_{h}^{l-} \cdot n_{h}^{l-})\\ 
&+ (Pn_{h}^{l-}-n^{-})\cdot \nabla_{\Gamma}u_{h}^{l-} \Big)\ ds.   
\end{align*}
For the first component of the above, we have
\begin{align*}
&\sum_{e_{h}^{l} \in \mathcal{E}_{h}^{l}}\int_{e_{h}^{l}}\frac{1}{\delta_{e_{h}}}(v_{h}^{l+}-v_{h}^{l-})\frac{1}{2}(n^{+}-Pn_{h}^{l+})\cdot \nabla_{\Gamma}u_{h}^{l+}\ ds\\ 
&\leq \norm{v_{h}^{l}}_{DG}\left(\sum_{e_{h}^{l} \in \mathcal{E}_{h}^{l}}\int_{e_{h}^{l}}\frac{1}{4}\frac{1}{(\delta_{e_{h}})^{2}}h_{e_{h}^{l}}\Big((n^{+}-Pn_{h}^{l+})\cdot \nabla_{\Gamma} u_{h}^{l+}\Big)^{2}\ ds\right)^{\frac{1}{2}} 
\end{align*}
after applying Cauchy-Schwartz. Using similar arguments as done for proving boundedness of the classical IP method (see \cite{arnold2002unified}), we have
\begin{align*}
&\sum_{e_{h}^{l} \in \mathcal{E}_{h}^{l}}\int_{e_{h}^{l}}\frac{1}{4}\frac{1}{(\delta_{e_{h}})^{2}}h_{e_{h}^{l}}\Big((n^{+}-Pn_{h}^{l+})\cdot \nabla_{\Gamma} u_{h}^{l+}\Big)^{2}\ ds \leq \sum_{e_{h}^{l} \in \mathcal{E}_{h}^{l}}\int_{e_{h}^{l}}\frac{1}{4}\frac{1}{(\delta_{e_{h}})^{2}}h_{e_{h}^{l}}|n^{+}-Pn_{h}^{l+}|^{2} |\nabla_{\Gamma} u_{h}^{l+}|^{2}\ ds \\
&\leq C\max_{e_{h}^{l} \in \mathcal{E}_{h}^{l}}\norm{n^{+}-Pn_{h}^{l+}}_{L^{\infty}(e_{h}^{l})}^{2} \sum_{K_{h}^{l} \in \mathcal{T}_{h}^{l}}\sum_{e_{h}^{l} \in \partial K_{h}^{l}} h_{e_{h}^{l}} \norm{\left. \nabla_{\Gamma}u_{h}^{l} \right |_{K_{h}^{l}}}_{L^{2}(e_{h}^{l})}^{2} \leq  C\max_{e_{h}^{l} \in \mathcal{E}_{h}^{l}}\norm{n^{+}-Pn_{h}^{l+}}_{L^{\infty}(e_{h}^{l})}^{2}\norm{u_{h}^{l}}_{DG}^{2}
\end{align*}
where the last inequality was derived using similar arguments as in the proof of Lemma \ref{thm:traceEstimatesLift}. For the second component, we have
\begin{align*}
&\sum_{e_{h}^{l} \in \mathcal{E}_{h}^{l}}\int_{e_{h}^{l}}\frac{1}{\delta_{e_{h}}}(v_{h}^{l+}-v_{h}^{l-})\frac{1}{2}(dH \nabla_{\Gamma}u_{h}^{l+} \cdot n_{h}^{l+})\ ds\\ 
&\leq \norm{v_{h}^{l}}_{DG}\left(\sum_{e_{h}^{l} \in \mathcal{E}_{h}^{l}}\int_{e_{h}^{l}}\frac{1}{4}\frac{1}{(\delta_{e_{h}})^{2}}h_{e_{h}^{l}}\Big(d H \nabla_{\Gamma}u_{h}^{l+}\cdot n_{h}^{l+}\Big)^{2}\ ds\right)^{\frac{1}{2}}. 
\end{align*}
Pursuing the analysis as before and using the fact that the Hessian $H$ is symmetric and bounded, we have  
\begin{align*}
&\sum_{e_{h}^{l} \in \mathcal{E}_{h}^{l}}\int_{e_{h}^{l}}\frac{1}{4}\frac{1}{(\delta_{e_{h}})^{2}}h_{e_{h}^{l}}\Big(d H \nabla_{\Gamma}u_{h}^{l+}\cdot n_{h}^{l+}\Big)^{2}\ ds = \sum_{e_{h}^{l} \in \mathcal{E}_{h}^{l}}\int_{e_{h}^{l}}\frac{1}{4}\frac{1}{(\delta_{e_{h}})^{2}}h_{e_{h}^{l}}d^{2} |H \nabla_{\Gamma}u_{h}^{l+}\cdot n_{h}^{l+}|^{2}\ ds \\
&= \sum_{e_{h}^{l} \in \mathcal{E}_{h}^{l}}\int_{e_{h}^{l}}\frac{1}{4}\frac{1}{(\delta_{e_{h}})^{2}}h_{e_{h}^{l}}d^{2} |\nabla_{\Gamma}u_{h}^{l+}\cdot H n_{h}^{l+}|^{2}\ ds \leq \sum_{e_{h}^{l} \in \mathcal{E}_{h}^{l}}\int_{e_{h}^{l}}\frac{1}{4}\frac{1}{(\delta_{e_{h}})^{2}}h_{e_{h}^{l}}d^{2} |\nabla_{\Gamma}u_{h}^{l+}|^{2} |H P n_{h}^{l+}|^{2}\ ds \\ 
&\leq  C\norm{d}_{L^{\infty}(\Gamma)}^{2} \sum_{K_{h}^{l} \in \mathcal{T}_{h}^{l}}\sum_{e_{h}^{l} \in \partial K_{h}^{l}} h_{e_{h}^{l}} \norm{\left. \nabla_{\Gamma}u_{h}^{l} \right |_{K_{h}^{l}}}_{L^{2}(e_{h}^{l})}^{2} \leq C \norm{d}_{L^{\infty}(\Gamma)}^{2} \norm{u_{h}^{l}}_{DG}^{2}
\end{align*}
where again the last inequality follows from applying similar arguments as in the proof of Lemma \ref{thm:traceEstimatesLift}. 

We can now estimate the error functional $E_{h}$:
\begin{align*}
|E_{h}(v_{h}^{l})| & \leq \norm{R_{h}-P}_{L^{\infty}(\Gamma)}\norm{u_{h}^{l}}_{DG}\norm{v_{h}^{l}}_{DG} + \norm{\frac{1}{\delta_{h}}-1}_{L^{\infty}(\Gamma)}\norm{u_{h}^{l}}_{DG}\norm{v_{h}^{l}}_{DG}\\
&+ \norm{1-\frac{1}{\delta_{h}}}_{L^{\infty}(\Gamma)}\norm{f}_{L^{2}(\Gamma)}\norm{v_{h}^{l}}_{DG} 
+ C \max_{e_{h}^{l} \in \mathcal{E}_{h}^{l}}\norm{n-P n_{h}^{l}}_{L^{\infty}(e_{h}^{l})}\norm{u_{h}^{l}}_{DG}\norm{v_{h}^{l}}_{DG}\\ 
&+  C \norm{d}_{L^{\infty}(\Gamma)}\norm{u_{h}^{l}}_{DG}\norm{v_{h}^{l}}_{DG}.
\end{align*}
So by Lemma \ref{Gamma2GammahSmall} and the stability estimate (\ref{eq:stabilityEstimateGamma2}) we have
\[|E_{h}(v_{h}^{l})| \leq C h^{2}\norm{f}_{L^{2}(\Gamma)}\norm{v_{h}^{l}}_{DG}\]
for every $v_{h}^{l} \in V_{h}^{l}$ as required.
\end{proof}

\section{Numerical Tests} \label{sec:numerics}
\subsection{Implementation Aspects} \label{sec:implaspect}
The IP method has been implemented using DUNE-FEM, a discretization module
based on the Distributed and Unified Numerics Environment (DUNE), 
(further information about DUNE can be found 
in  \cite{dunegridpaperI:08}, \cite{dunegridpaperII:08} and \cite{dune-web-page}). 
DG methods are well tested 
for the DUNE-FEM module, as shown in \cite{dunefempaper:10},
\cite{cdg2:10}, but only simple schemes have been tested for surface PDEs 
(further information about the DUNE-FEM module can be found in \cite{dunefempaper:10} and 
\cite{dunefem-web-page}). 
In all our numerical tests we choose 
the polynomial order on each element $K_{h} \in \mathcal{T}_{h}$ to be $1$
and interior penalty parameters to satisfy (\ref{penaltyParameterBound}). 
The initial mesh generation for each test case is performed using the 3D surface mesh generation module of the Computational Geometry Algorithms Library (CGAL) (see \cite{rineau20093d}).

When performing mesh refinements it is often the case that there is no explicit projection map for mapping newly created nodes from $\Gamma_{h}$ to $\Gamma$, hence $\xi(x)$ must be approximated. Two different algorithms, discussed in more detail in \cite{demlow2008adaptive}, have been tested for such problems: one being Newton's method and the other being an ad-hoc first-order method. Assume that $x_{0} \in U$ and that we wish to compute $\xi(x_{0})$. The Newton method seeks to find a stationary point of the function $F(x,\lambda) = |x-x_{0}|^{2} + \lambda \phi(x)$ with starting values chosen to be $(x_{0},\lambda_{0}) = (x_{0},2\phi(x_{0})/|\nabla \phi(x_{0})|^{2})$, where $\phi$ is the level-set function of $\Gamma$ (and not necessarily a signed-distance function). We iterate the method until
\begin{equation}\label{eq:stoppingCriterion}
\left(\frac{\phi(x)^2}{|\nabla \phi(x)|^{2}}+\left|\frac{\nabla \phi(x)}{|\nabla \phi(x)|} - \frac{x-x_{0}}{|x-x_{0}|}\right|^{2}\right)^{1/2} < tol
\end{equation}
is reached. Note that this stopping criteria incorporates both how close the iterate is to the surface $\Gamma$ as well as how accurately it lies in the normal direction. The second method is given by the following first-order algorithm:

1. Stipulate $tol$ and $x_{0}$ and initialise $x=x_{0}$.

2. While (\ref{eq:stoppingCriterion}) is not satisfied, iterate the following steps:

\hspace{3mm} (a) Calculate $\tilde{x} = x - \frac{\phi(x)\nabla \phi(x)}{|\nabla \phi(x)|^{2}}$ \ \mbox{and}\ $dist = sign(\phi(x_{0}))|\tilde{x} - x_{0}|$.

\hspace{3mm} (b) Set $x = x_{0} - dist \frac{\nabla \phi(\tilde{x})}{|\nabla \phi(\tilde{x})|}$.

We can make this algorithm more flexible by only requiring a second order finite difference approximation of $\nabla \phi(x)$. It was observed in \cite{demlow2008adaptive} that in practice the second of the two algorithms was more efficient due to the fact that each step of Newton's method is relatively expensive. This was observed in all of our numerical tests. It has also been noted that there has not been any rigorous error analysis done for either of the two algorithms with respect to the stopping criterion (\ref{eq:stoppingCriterion}). Though we do not provide any such error analysis, our numerical tests suggest that both algorithms may stagnate at certain points for which the stopping criterion tolerance is never reached, even with fairly refined initial meshes. The normal direction error contribution of the stopping criteria appears to be responsible for the algorithm stagnating, hence for such points we remove this contribution so that the algorithm terminates and the resulting point lies on the surface nevertheless (albeit not necessarily at $\xi(x_{0})$).

In addition, we make use of this algorithm to provide a generic implementation of test problems on surfaces. Computing the Laplace-Beltrami operator of some given function over an arbitrary compact smooth connected and oriented surface given by the zero level-set of some function is tedious and requires changing the implementation for every such surface. In particular, we would need to explicitly compute the outward unit normal of the surface and its gradient whenever we consider a new surface. For any $u \in C^{2}(\mathbb{R}^{3})$, we have

\begin{equation}\label{eq:LaplaceBeltramiExpression}
\Delta_{\Gamma} u = \Delta u - \nu \cdot \nabla^{2}u \nu - \mbox{tr}(\nabla \nu) \nabla u \cdot \nu
\end{equation}
where $\Delta$ is the usual Euclidean Laplace operator in $\mathbb{R}^{3}$, $\nabla^{2}u \in \mathbb{R}^{3 \times 3}$ the (Euclidean) Hessian of u, $\nabla u$ the (Euclidean) gradient of $u$ and finally $\mbox{tr}(\nabla \nu)$ the trace of $\nabla \nu$ where $\nabla \nu \in \mathbb{R}^{3 \times 3}$ whose entries are the (Euclidean) partial derivatives of each component of the normal. We can make use of the ad-hoc first-order algorithm described previously to approximate the outward unit normal $\nu$ of $\Gamma$ in (\ref{eq:LaplaceBeltramiExpression}): this is done by computing $\nu(\xi(x_{0})) \approx sign(\phi(x_{0}))(\tilde{\xi}(x_{0}) - x_{0})$ where $\tilde{\xi}(x_{0})$ is the approximation of $\xi(x_{0})$ resulting from the algorithm . We may also approximate the (diagonal) entries of $\nabla \nu$ via second-order finite difference approximations as done for the approximation of $\nabla \phi$ in the first-order algorithm. Such a generic implementation has the benefit of only requiring input of the level-set function for the surface and nothing more, significantly facilitating numerical tests. The error caused by our approximation of the Laplace-Beltrami operator appears not to affect the resulting convergence order for any of our test cases.

\subsection{Approximation of Surface Conormals}
Consider the IP bilinear form $\tilde{a}_{\Gamma_{h}}^{IP}$ given by 
\begin{align}
\label{eq:InteriorPenaltyGammahFormNumerics}
\tilde{a}_{\Gamma_{h}}^{IP}(u_{h},v_{h}) &:= \sum_{K_{h} \in T_{h}}\int_{K_{h}}\nabla_{\Gamma_{h}}u_{h}\cdot \nabla_{\Gamma_{h}}v_{h} + u_{h} v_{h}\ dA_{h}\notag \\ 
- \sum_{e_{h} \in \mathcal{E}_{h}}\int_{e_{h}}&(u_{h}^{+}-u_{h}^{-})\frac{1}{2}(\nabla_{\Gamma_{h}}v_{h}^{+}\cdot n^{+}_{e_{h}} - \nabla_{\Gamma_{h}}v_{h}^{-}\cdot n^{-}_{e_{h}}) + (v_{h}^{+}-v_{h}^{-})\frac{1}{2}(\nabla_{\Gamma_{h}}u_{h}^{+}\cdot n^{+}_{e_{h}} - \nabla_{\Gamma_{h}}u_{h}^{-}\cdot n^{-}_{e_{h}})\ ds_{h}\notag \\
+ \sum_{e_{h} \in \mathcal{E}_{h}}\int_{e_{h}}&\beta_{e_{h}}(u_{h}^{+}-u_{h}^{-})(v_{h}^{+}-v_{h}^{-})\ ds_{h}
\end{align}
where $n_{e_{h}}^{+}$ and $n_{e_{h}}^{-}$ are simply vectors which lie on the intersection $e_{h} \in \mathcal{E}_{h}$ of neighbouring elements $K_{h}^{+}$ and $K_{h}^{-}$. 
Now assume that we want to assemble the system matrix on an element $K_{h}$ and we assume that $K_{h}=K_{h}^-$ for all $e_h\subset \partial K_h$. 
To this end, we fix $v_{h} = \varphi^{-}$ with $\mbox{supp}(\varphi^{-}) = K_{h}$ 
which leads to
\begin{align*}
&\tilde{a}_{\Gamma_{h}}^{IP}(u_{h},\varphi^{-}) := \int_{K_{h}}\nabla_{\Gamma_{h}}u_{h}\cdot \nabla_{\Gamma_{h}}\varphi^{-} + u_{h} \varphi^{-}\ dA_{h} \\ 
&+ \sum_{e_{h} \subset \partial K_{h}}\int_{e_{h}}(u_{h}^{+}-u_{h}^{-})\frac{1}{2}\nabla_{\Gamma_{h}}\varphi^{-}\cdot n^{-}_{e_{h}} + \varphi^{-}\frac{1}{2}(\nabla_{\Gamma_{h}}u_{h}^{+}\cdot n^{+}_{e_{h}} - \nabla_{\Gamma_{h}}u_{h}^{-}\cdot n^{-}_{e_{h}})\ ds_{h} \\
&- \sum_{e_{h} \subset \partial K_{h}}\int_{e_{h}}\beta_{e_{h}}(u_{h}^{+}-u_{h}^{-})\varphi^{-}\ ds_{h}.
\end{align*}
To assemble the block on the diagonal of the matrix we need to take
$u_h=\psi^-$ with $\mbox{supp}(\psi^-) = K_{h}$. For the off-diagonal block
we take $u_h=\psi^+$ with $\mbox{supp}(\psi^+) = K_{h}^{+}$ for one neighbour $K_{h}^{+}$ of $K_h$. 
We will then discuss different choices for $n_{e_{h}}^{+/-}$ which are linked to several intuitive ways of approximating respectively the surface conormals $n^{+/-}$. We use one choice
for $n^{+}_{e_{h}}$ in both cases. To cover all of the
choices we want to consider, it is necessary to use different choices for $n^{-}_{e_{h}}$, i.e., the vector belonging to the element
$K_{h}$ on which we are assembling the matrix. For the diagonal block we
will denote our choice for this vector with $n_D^{-}$ and use the original
notation $n^-_{e_{h}}$ for the choice used to assemble the off-diagonal
block.
Note that $n_{D}^{-} = n_{h}^{-}$ for all of the choices discussed below
except for Choice 3. 

Now consider $u_h=\psi^-$ with $\mbox{supp}(\psi^-) = K_{h}$ in
\eqref{eq:InteriorPenaltyGammahFormNumerics} using $n_D^{-}$ instead of
$n^-_{e_{h}}$:
\begin{align*}
&\tilde{a}_{\Gamma_{h}}^{IP}(\psi^{-},\varphi^{-}) :=
\int_{K_{h}}\nabla_{\Gamma_{h}}\psi^{-}\cdot \nabla_{\Gamma_{h}}\varphi^{-}
+ \psi^{-} \varphi^{-}\ dA_{h} \\ 
&- \sum_{e_{h} \subset \partial K_{h}}\int_{e_{h}}\frac{1}{2}\psi^{-}\nabla_{\Gamma_{h}}\varphi^{-}\cdot n^{-}_{D} + \varphi^{-}\frac{1}{2} \nabla_{\Gamma_{h}}\psi^{-}\cdot n^{-}_{D} - \beta_{e_{h}}\psi^{-}\varphi^{-}\ ds_{h}.
\end{align*}
Next we take $u_h=\psi^+$ with
$\mbox{supp}(\psi^+) = K_{h}^{+}$ for one neighbour $K_{h}^{+}$ of $K_h$, we now have
\begin{align*}
\tilde{a}_{\Gamma_{h}}^{IP}(\psi^{+},\varphi^{-}) := &\sum_{e_{h} \subset \partial K_{h}}\int_{e_{h}}\frac{1}{2}\psi^{+}\nabla_{\Gamma_{h}}\varphi^{-}\cdot n^{-}_{e_{h}}\ ds_{h} + \varphi^{-}\frac{1}{2} \nabla_{\Gamma_{h}}\psi^{+}\cdot n^{+}_{e_{h}}\ ds_{h} - \beta_{e_{h}}\psi^{+}\varphi^{-}\ ds_{h}.
\end{align*}
We can now prescribe choices for the vectors $n_{D}^{-}$, $n_{e_{h}}^{-}$, $n_{e_{h}}^{+}$  and will later investigate the behaviour of the numerical scheme (\ref{eq:InteriorPenaltyGammahFormNumerics}) for different choices of these three vectors.

\textbf{Choice 1}
\[n_{D}^{-} = n_{h}^{-}\   \ , \   \ n_{e_{h}}^{-} = n_{h}^{-}\   \ , \   \ n_{e_{h}}^{+} = -n_{h}^{-}.\]
Such a choice corresponds to using the IP method in a planar setting, for which $n_{h}^{+} = -n_{h}^{-}$, and is the simplest scheme to implement.

\textbf{Choice 2}
\[n_{D}^{-} = n_{h}^{-}\   \ , \   \ n_{e_{h}}^{-} = n_{h}^{-}\   \ , \   \ n_{e_{h}}^{+} = n_{h}^{+}.\]
This choice yields the numerical scheme (\ref{eq:InteriorPenaltyGammahForm}) that has been discussed up to now and used in the error analysis.

\textbf{Choice 3}
\[n_{D}^{-} = \frac{\frac{1}{2}(n_{h}^{-}-n_{h}^{+})}{|\frac{1}{2}(n_{h}^{-}-n_{h}^{+})|}\   \ , \   \ n_{e_{h}}^{-} = \frac{\frac{1}{2}(n_{h}^{-}-n_{h}^{+})}{|\frac{1}{2}(n_{h}^{-}-n_{h}^{+})|}\   \ , \   \ n_{e_{h}}^{+} = \frac{\frac{1}{2}(n_{h}^{+}-n_{h}^{-})}{|\frac{1}{2}(n_{h}^{+}-n_{h}^{-})|}.\]
This choice corresponds to prescribing the vectors to be the average of the two conormals and yields additional symmetry in the resulting matrix due to the fact that the vectors are now independent of the element on which they are computed.

\textbf{Choice 4}
\[n_{D}^{-} = n_{h}^{-}\   \ , \   \ n_{e_{h}}^{-} = -n_{h}^{+}\   \ , \   \ n_{e_{h}}^{+} = -n_{h}^{-}.\]
This particular choice corresponds to using the  formulation of the IP method found for example in~\cite{arnold2002unified} on $\Gamma_{h}$, but with a \emph{modified} penalty term that does not depend on the conormals i.e.
\begin{align*} 
\tilde{a}_{\Gamma_{h}}^{IP}(u_{h},v_{h}) &= \sum_{K_{h} \in \mathcal{T}_{h}}\int_{K_{h}}\nabla_{\Gamma_{h}}u_{h}\cdot \nabla_{\Gamma_{h}}v_{h} + u_{h} v_{h}\ dA_{h} \\ 
&- \sum_{e_{h}\subset \partial K_{h}}\int_{e_{h}}(u_{h}^{+}n_{h}^{+} + u_{h}^{-}n_{h}^{-})\cdot \frac{1}{2}(\nabla_{\Gamma_{h}}v_{h}^{+} + \nabla_{\Gamma_{h}}v_{h}^{-}) + \frac{1}{2}(\nabla_{\Gamma_{h}}u_{h}^{+}+\nabla_{\Gamma_{h}}u_{h}^{-})\cdot (v_{h}^{+}n_{h}^{+} + v_{h}^{-}n_{h}^{-})\ ds_{h} \\
&+ \sum_{e_{h}\subset \partial K_{h}}\int_{e_{h}}\beta_{e_{h}}(u_{h}^{+}-u_{h}^{-})(v_{h}^{+}-v_{h}^{-})\ ds_{h}\ \ \mbox{(modified penalty term)}.
\end{align*}
We summarise the choices in Table \ref{tab:normalChoiceTable}.
\begin{table}[h!]
\begin{center}
\small
\begin{tabular}{|c|c|c|c|c|}
\hline
\mbox{Choice}&$n_{D}^{-}$&$n_{e_{h}}^{-}$&$n_{e_{h}}^{+}$&\mbox{Description}\\
\hline
1&$n_{h}^{-}$&$n_{h}^{-}$&$-n_{h}^{-}$&\mbox{Planar (non-sym)}\\
2&$n_{h}^{-}$&$n_{h}^{-}$&$n_{h}^{+}$&\mbox{Analysis (sym pos-def)}\\
3&$\frac{\frac{1}{2}(n_{h}^{-}-n_{h}^{+})}{|\frac{1}{2}(n_{h}^{-}-n_{h}^{+})|}$&$\frac{\frac{1}{2}(n_{h}^{-}-n_{h}^{+})}{|\frac{1}{2}(n_{h}^{-}-n_{h}^{+})|}$&$\frac{\frac{1}{2}(n_{h}^{+}-n_{h}^{-})}{|\frac{1}{2}(n_{h}^{+}-n_{h}^{-})|}$&\mbox{Average (sym pos-def)}\\
4&$n_{h}^{-}$&$-n_{h}^{+}$&$-n_{h}^{-}$&\mbox{\cite{arnold2002unified} (sym pos-def)}\\
\hline
\end{tabular}
\end{center}
\caption{Choices of $n_{D}^{-}$, $n_{e_{h}}^{+}$ and $n_{e_{h}}^{-}$, description of the numerical schemes they respectively lead to\\ and properties of resulting matrix.}
\label{tab:normalChoiceTable}
\end{table}

We also consider the \cite{arnold2002unified} formulation with its true penalty term given by
\[\sum_{e_{h}\subset \partial K_{h}}\int_{e_{h}}\beta_{e_{h}}(u_{h}^{+}n_{h}^{+} + u_{h}^{-}n_{h}^{-}) \cdot (v_{h}^{+}n_{h}^{+} + v_{h}^{-}n_{h}^{-})\ ds_{h}\ \ \mbox{(true penalty term)}.\]
Choosing $v_{h} = \varphi^{-}$ and $u_{h} = \psi^{-}$ as before yields
\[\sum_{e_{h}\subset \partial K_{h}}\int_{e_{h}}\beta_{e_{h}}\psi^{-}\varphi^{-}\ ds_{h}.\]
For $u_{h} = \psi^{+}$ we now have,
\[\sum_{e_{h}\subset \partial K_{h}}\int_{e_{h}}\beta_{e_{h}}\psi^{+}\varphi^{-} (n_{h}^{+} \cdot n_{h}^{-})\ ds_{h}.\]

The matrices arising from Choices 2-4 are symmetric positive definite, so the Conjugate Gradient (CG) method is particularly well suited for such matrix problems. Choice 1 however yields a non-symmetric matrix, for which we use the Biconjugate Gradient Stabilized (BICGSTAB) method. All of these solvers make use of the algebraic multigrid algorithm (AMG) preconditioner coupled with the incomplete-LU factorisation preconditioner to speed up the solvers. Information on the implementation of these solvers and preconditioners in DUNE can be found in \cite{ISTL} and on their parallelisation in \cite{ISTLParallel}. 

We first tested our code on a sphere where the projection algorithm
described in Section \ref{sec:implaspect} is not required. The results showed that the expected
convergence rates and the choices of $n_{D}^{-}$, $n_{e_{h}}^{-}$ and $n_{e_{h}}^{+}$ had little influence on
the results. Hence we have decided not to include these tests since no
insight can be gained.

\subsection{Test Problem on Dziuk Surface}
The first test problem, taken from \cite{dziuk1988finite}, considers
\begin{equation}\label{eq:testProblem} 
-\Delta_{\Gamma} u + u = f
\end{equation}
on the surface $\Gamma = \{ x \in \mathbb{R}^{3}\ : \ (x_{1}-x_{3}^{2})^{2}+x_{2}^{2}+x_{3}^{2}= 1 \}$ whose exact solution is chosen to be given by $u(x)=x_{1}x_{2}$. The outward unit normal to this surface is given by $\nu(x) = (x_{1} - x_{3}^{2},x_{2},x_{3}(1 - 2(x_{1} - x_{3}^{2}))) / (1 + 4x_{3}^{2}(1-x_{1}-x_{2}^{2}))^{1/2}$. There is no explicit projection map for mapping newly created nodes to $\Gamma$ so $\xi(x)$ has to be approximated via the ad-hoc first order algorithm described in Section \ref{sec:implaspect}.
\begin{table}[h!]
\begin{center}
\begin{tabular}{|c|c|c|c|c|c|}
\hline
\mbox{Elements}&$h$&$L_{2} \mbox{-error}$&$L_{2}\mbox{-eoc}$&$DG\mbox{-error}$&$DG\mbox{-eoc}$\\
\hline
92&0.704521&0.243493&&0.894504&\\
368&0.353599&0.0842372&1.53&0.490805&0.87\\
1472&0.176993&0.0268596&1.65&0.263808&0.90\\
5888&0.0885231&0.00637826&2.07&0.135162&0.97\\
23552&0.0442651&0.00171047&1.90&0.0685366&0.98\\
94208&0.022133&0.000416366&2.04&0.0343677&1.00\\
376832&0.0110666&0.000104274&2.00&0.0171891&1.00\\
1507328&0.0055333&2.60734e-05&2.00&0.0085935&1.00\\
\hline
\end{tabular}
\end{center}
\caption{Errors and convergence orders for (\ref{eq:testProblem}) on the Dziuk surface with Choice 2 (analysis).}
\label{tab:ConformGridTableDziuk}
\end{table}
  
Table \ref{tab:ConformGridTableDziuk} shows the $L^{2}$ and $DG$ errors for Choice 2. As expected, the experimental orders of convergence (EOCs) match up well with the theoretical convergence rates. Figure \ref{fig:ConformGridFigure1Dziuk} shows the resulting DG approximation to (\ref{eq:testProblem}) on the Dziuk surface using Choice 2.
\begin{figure}[htp]
\centering
\subfloat[]{\includegraphics[width=0.499 \textwidth]{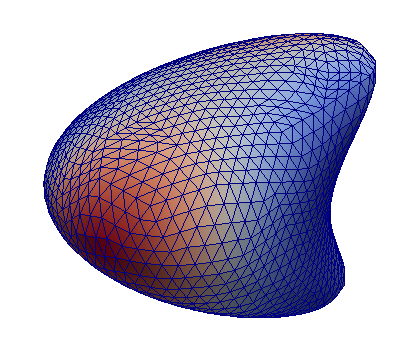}}\\
\caption{DG approximation of (\ref{eq:testProblem}) on the Dziuk surface with Choice 2 (analysis).}
\label{fig:ConformGridFigure1Dziuk}
\subfloat[]{\includegraphics[width=0.499 \textwidth]{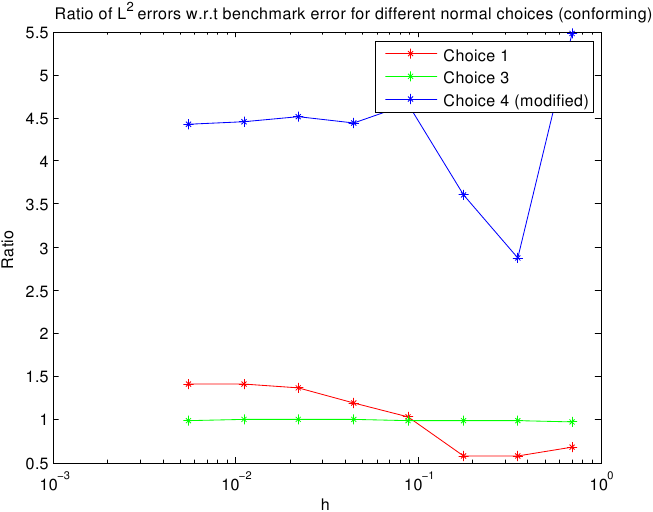}}
\subfloat[]{\includegraphics[width=0.482 \textwidth]{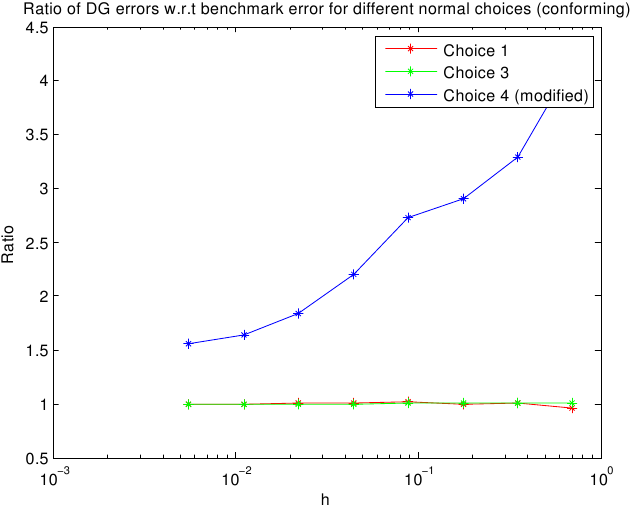}}
\caption{Ratio of respectively $L^{2}$ and $DG$ errors for (\ref{eq:testProblem}) on the Dziuk surface with respect to\\ the analysis error (Choice 2) for Choices 1, 3 and 4.}
\label{fig:ConformGridFigure2Dziuk}
\end{figure}

Figure \ref{fig:ConformGridFigure2Dziuk}(a,b) shows respectively the
ratios of the $L^{2}$ and $DG$ errors
$\frac{\mbox{Err}_{i}}{\mbox{Err}_{2}}$ with $i = 1,3,4$ where $\mbox{Err}_{i}$ denotes the error in the corresponding norm when using Choice $i$. Choices 2 (analysis) and 3 (average) appear to give the best results in
both the $L^{2}$ and $DG$ norms. In particular, the additional symmetry induced by using Choice 3 which we mentioned previously makes it the preferable choice.

A few remarks on Choice 4 with the \emph{true} penalty term which, 
as mentioned before, would correspond to the \cite{arnold2002unified} IP method on $\Gamma_{h}$: 
interestingly, the scheme fails to converge for such a choice. 
The numerical scheme appears to be particularly sensitive to small perturbations in the off-diagonal 
entries of the resulting matrix, namely the ones caused by the product of the conormals 
$n_{h}^{+} \cdot n_{h}^{-}$ when using the true penalty term for Choice 4. 
Note that in the flat case, $n_{h}^{+} \cdot n_{h}^{-}$ is equal to
$-1$. We tried to reproduce this problem in the flat case, taking two
different values for the penalty parameter on $e_h$ depending on whether we are
assembling the diagonal or the off-diagonal block. Already a factor of
$10^{-5}$ leads to similar problems with stability.
Since Choice 4 with or without the true penalty term was always less
accurate than the other choices, we omit this choice in our next test case.

\subsection{Test Problem on Enzensberger-Stern Surface}
Our next test problem considers (\ref{eq:testProblem}) on  $\Gamma = \{ x \in \mathbb{R}^{3}\ : \ 400(x^2y^2 + y^2z^2 + x^2z^2) - (1-x^2-y^2-z^2)^3 - 40 = 0 \}$ whose exact solution is again chosen to be given by $u(x)=x_{1}x_{2}$. As for the previous test problem, there is no explicit projection map so we make use of the first order ad-hoc algorithm. In this test problem, the computation of $\Delta_{\Gamma} u$ to derive the right-hand side of  (\ref{eq:testProblem}) is done via our approximation of the Laplace-Beltrami operator described in Section \ref{sec:implaspect}.
\begin{table}[h!]
\begin{center}
\begin{tabular}{|c|c|c|c|c|c|}
\hline
\mbox{Elements}&$h$&$L_{2} \mbox{-error}$&$L_{2}\mbox{-eoc}$&$DG\mbox{-error}$&$DG\mbox{-eoc}$\\
\hline
2358&0.163789&0.476777&&0.998066&\\
9432&0.0817973&0.175293&1.44&0.472241&1.08\\
37728&0.040885&0.0160606&3.45&0.150144&1.65\\
150912&0.0204411&0.00139698&3.52&0.0703901&1.09\\
603648&0.0102204&0.00033846&2.04&0.03473453&1.02\\
2414592&0.00511&7.86713e-05&2.10&0.0172348&1.01\\
\hline
\end{tabular}
\end{center}
\caption{Errors and convergence orders for (\ref{eq:testProblem}) on the Enzensberger-Stern surface with Choice 2 \\ (analysis).}
\label{tab:ConformGridTableEnzensburgerStern}
\end{table}

Table \ref{tab:ConformGridTableEnzensburgerStern} shows the $L^{2}$ and
$DG$ errors for Choice 2.
Although the EOCs are more erratic than for the previous test
problem, largely due to our approximation of the Laplace-Beltrami operator,
they nevertheless match up well with theoretical convergence rates. Figure
\ref{fig:ConformGridFigure1EnzensburgerStern} shows the resulting DG
approximation to (\ref{eq:testProblem}) on this surface using Choice 2.
\begin{figure}[htp]
\centering
\subfloat[]{\includegraphics[width=0.499 \textwidth]{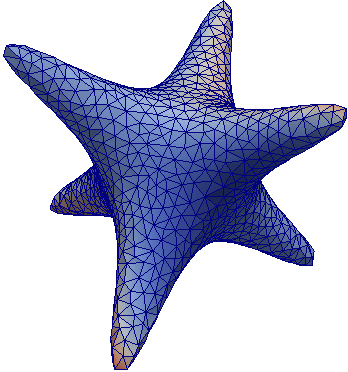}}\\
\caption{DG approximation of (\ref{eq:testProblem}) on the Enzensberger-Stern surface with Choice 2 (analysis).}
\label{fig:ConformGridFigure1EnzensburgerStern}
\subfloat[]{\includegraphics[width=0.499 \textwidth]{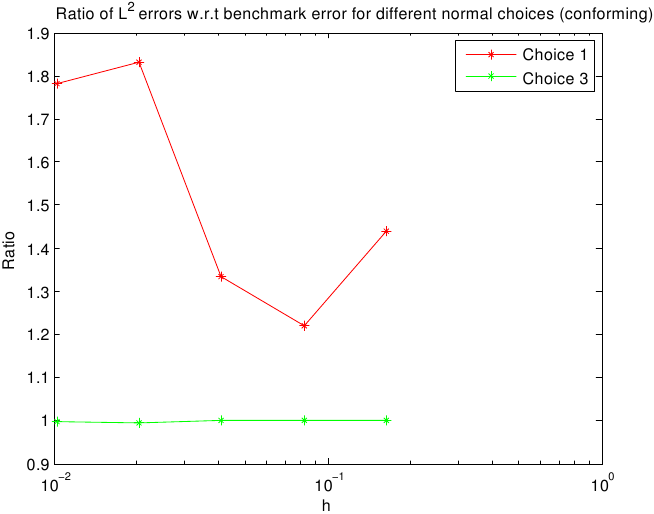}}
\subfloat[]{\includegraphics[width=0.49 \textwidth]{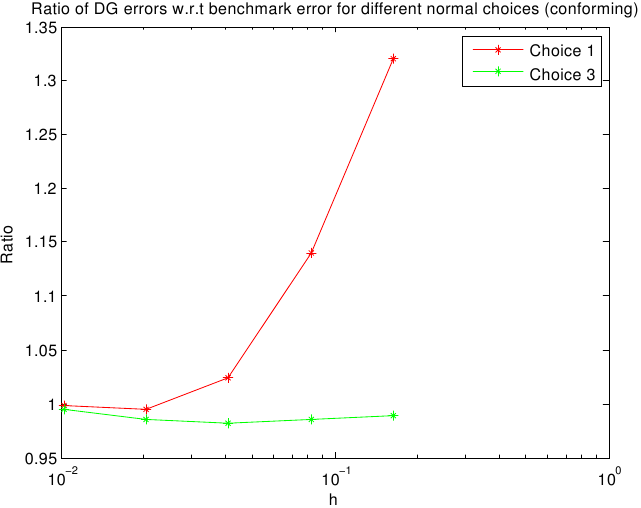}}
\caption{Ratio of respectively $L^{2}$ and $DG$ errors for (\ref{eq:testProblem}) on the Enzensberger-Stern surface with\\ respect to the analysis error (Choice 2) for Choices 1 and 3.}
\label{fig:ConformGridFigure2EnzensburgerStern}
\end{figure}
We again consider the DG approximation of (\ref{eq:testProblem}) for different choices of $n_{D}^{-}$, $n_{e_{h}}^{+}$ and $n_{e_{h}}^{-}$. Figure \ref{fig:ConformGridFigure2EnzensburgerStern}(a,b) shows respectively the ratios of the $L^{2}$ and $DG$ errors. These results confirm that Choices 2 and 3 are the preferable ones to use for DG schemes on surfaces.

\section{Extensions}
Although our analysis was restricted to conforming grids due to the nature of the surface approximation, our numerical tests suggest that the estimates of Theorem \ref{aprioriErrorEstimateIP} also hold for nonconforming grids as shown in Table \ref{tab:NonConformGridTableDziuk} for the Dziuk surface. 
\begin{table}[htp]
\begin{center}
\begin{tabular}{|c|c|c|c|c|c|}
\hline
\mbox{Elements}&$h$&$L_{2} \mbox{-error}$&$L_{2}\mbox{-eoc}$&$DG\mbox{-error}$&$DG\mbox{-eoc}$\\
\hline
230&0.353599&0.21889& &0.777436&\\
920&0.176993&0.0530078&2.05&0.413817&0.91\\
3680&0.0885231&0.0281113&0.92&0.223119&0.89\\
14720&0.0442651&0.00442299&2.67&0.111518&1.00\\
58880&0.022133&0.00104207&2.08&0.0562128&0.99\\
235520&0.0110666&0.00026444&1.99&0.0281247&1.00\\
942080&0.00553329&6.60383e-05&2.00&0.0140544&1.00\\
\hline
\end{tabular}
\end{center}
\caption{Errors and convergence orders for (\ref{eq:testProblem}) on the Dziuk surface with Choice 2 (analysis) for\\ a nonconforming grid.}
\label{tab:NonConformGridTableDziuk}
\end{table}
Future work aims to derive a-priori error estimates for nonconforming grids. 
 
\cite{demlow2009higher} has proven that in particular, for a linear approximation of the surface and quadratic polynomial basis functions, the FEM error scales quadratically in both the $L^{2}$ and $H^{1}$ norms. Numerical tests suggest that our DG scheme scales similarly in the $L^{2}$ and $DG$ norms as shown in Table \ref{tab:ConformGridTableDziukQuadratic} for the Dziuk surface. In future work we aim to derive higher order a-priori error estimates (that is, both higher order polynomial basis functions and higher order approximations of the surface) for the DG approximation in a similar fashion to the work done in \cite{demlow2009higher}.       
\begin{table}[htp]
\begin{center}
\begin{tabular}{|c|c|c|c|c|c|}
\hline
\mbox{Elements}&$h$&$L_{2} \mbox{-error}$&$L_{2}\mbox{-eoc}$&$DG\mbox{-error}$&$DG\mbox{-eoc}$\\
\hline
92&0.704521&0.136442&&0.322416&\\
368&0.353599&0.0551454&1.31&0.150303&1.10\\
1472&0.176993&0.0215041&1.36&0.0601722&1.32\\
5888&0.0885231&0.00448861&2.26&0.0182412&1.72\\
23552&0.0442651&0.00120287&1.90&0.00513161&1.83\\
94208&0.022133&0.00029651&2.02&0.00130482&1.98\\
376832&0.0110666&7.41044e-05&2.00&0.00032728&2.00\\
\hline
\end{tabular}
\end{center}
\caption{Errors and convergence orders for (\ref{eq:testProblem}) on the Dziuk surface with Choice 3 (average)\\ using quadratic polynomial basis functions.}
\label{tab:ConformGridTableDziukQuadratic}
\end{table}

\section*{Acknowledgements}
This research has been supported by the British Engineering and Physical Sciences Research Council (EPSRC), Grant EP/H023364/1.

\bibliographystyle{IMANUM-BIB}
\bibliography{IMANUM-refs}

\end{document}